\documentclass[12pt]{amsart}
\usepackage[latin1]{inputenc}
\usepackage{mathptmx}
\usepackage{amscd}
\usepackage{amssymb}
\newtheorem{theorem}{Theorem}
\newtheorem{lemma}[theorem]{Lemma}
\newtheorem{corollary}[theorem]{Corollary}

\theoremstyle{definition}
\newtheorem{remark}[theorem]{Remark}
\newtheorem{definition}[theorem]{Definition}

\let\phi=\varphi

\def\Ann{\operatorname{Ann}}
\def\Ass{\operatorname{Ass}}
\def\Min{\operatorname{Min}}

\let\oldbigwedge\bigwedge
\def\BIGwedge{{\textstyle\oldbigwedge}}
\def\medwedge{{\scriptstyle\oldbigwedge}}
\def\bigwedge{\mathchoice{\BIGwedge}{\BIGwedge}{\medwedge}{}}

\DeclareMathOperator{\Jac}{Jac}
\DeclareMathOperator{\Nil}{Nil}
\DeclareMathOperator{\rad}{rad}
\DeclareMathOperator{\FId}{FId}
\DeclareMathOperator{\Max}{Max}

\let\epsilon=\varepsilon

\begin{document}
\title{Content Algebras over Commutative Rings with Zero-Divisors}

\author{Peyman Nasehpour}
\address{Universit\"at Osnabr\"uck, FB Mathematik/Informatik, 49069
Osnabr\"uck, Germany} \email{pnasehpo@uni-osnabrueck.de} \email{nasehpour@gmail.com}

\begin{abstract}
Let $M$ be an $R$-module and $c$ the function from $M$ to the ideals of $R$ defined by $c(x) = \cap \lbrace I \colon I \text{~is an ideal of~} R \text{~and~} x \in IM \rbrace $. $M$ is said to be a content $R$-module if $x \in c(x)M $, for all $x \in M$. $B$ is called a content $R$-algebra, if it is a faithfully flat and content $R$-module and it satisfies the Dedekind-Mertens content formula. In this article, we prove some new results for content modules and algebras by using ideal theoretic methods.
\end{abstract}

\maketitle

\tableofcontents

\section{Introduction}

This preprint is a collection of some of the definitions and results that the author was working on them during his studentship at the Department of Mathematics and Computer Science at the University of Osnabr\"uck. Though this preprint was never submitted for publication, but its definitions and results were reproduced and used in a couple of different papers composed by the author and were published in different journals. Since the author has cited to this preprint in some of his papers, it was necessary to update and edit it for the convenience of the readers of his papers.

$ $

Throughout this paper, all rings are commutative with unit and all modules are assumed to be unitary. In this paper, we discuss the ideal theoretic properties of some special algebras called content algebras. This concept stems from \textit{Dedekind-Mertens content formula} for polynomial rings. For doing this, we need to know about content modules introduced in \cite{OR}, which in the next section, we introduce them and then we start the main theme of the paper that is on content algebras over rings with zero-divisors. Our main goal is to show that a couple of ideal theoretic results of polynomial rings, also hold for content algebras.

$ $

The class of content modules are themselves considerable and interesting in the field of module theory. For example, all projective modules are content and a kind of Nakayama lemma holds for content modules.

Let $R$ be a commutative ring with identity, and $M$ a unitary $R$-module and the \textit{content function}, $c$ from $M$ to the ideals of $R$ defined by
$$
c(x) = \bigcap \lbrace I \colon I \text{~is an ideal of~} R \text{~and~} x \in IM \rbrace.
$$

$M$ is called a \textit{content $R$-module} if $x \in c(x)M $, for all $x \in M$.

$ $

In Section 2, we prove that if $M$ is a content $R$-module and $\Jac(R)$ is the Jacobson radical of $R$ and $I$, an ideal of $R$ such that $I \subseteq \Jac(R)$, then $IM = M$ implies $M = (0)$. Also we introduce content and weak content algebras and mention some of their basic properties that we need them in the rest of the paper for the convinience of the reader. Let $R$ be a commutative ring with identity and $R^\prime$ an $R$-algebra. $R^\prime$ is defined to be a \textit{content $R$-algebra}, if the following conditions hold:

\begin{enumerate}
 \item
$R^\prime$ is a content $R$-module.
 \item
(\textit{Faithfully flatness}) For any $r \in R$ and $f \in R^\prime$, the equation $c(rf) = rc(f)$ holds and $c(R^\prime) = R$.
 \item
(\textit{Dedekind-Mertens content formula}) For each $f$ and $g$ in $R^\prime$, there exists a natural number $n$ such that $c(f)^n c(g) = c(f)^{n-1} c(fg)$.
\end{enumerate}

In this section, also we prove that if $R$ is a ring and $S$, a commutative monoid, then the monoid ring $B=R[S]$ is a content $R$-algebra if and only if one of the following conditions satisfies:

\begin{enumerate}
\item For $f,g \in B$, if $c(f) = c(g) = R$, then $c(fg) = R$.
\item (\textit{McCoy's Property}) For $g \in B$, $g$ is a zero-divisor of $B$ iff there exists $r \in R-\lbrace 0 \rbrace$ such that $rg = 0$.
\item $S$ is a cancellative and torsion-free monoid.
\end{enumerate}

In Section 3, we discuss about prime ideals of content and weak content algebras (Cf. \cite{R}) and we show that in content extensions, minimal primes extend to minimal primes. More precisely, if $B$ is a content $R$-algebra, then there is a correspondence between $\Min(R)$ and $\Min(B)$, with the function $ \phi : \Min(R) \longrightarrow \Min(B)$ defined by $\underline{p} \longrightarrow \underline{p}B$.

$ $

In Section 4, we introduce a family of rings and modules who have very few zero-divisors. It is a well-known result that the set of zero-divisors of a finitely generated module over a Noetherian ring is a finite union of the associated primes of the module \cite[p. 55]{K}. Rings having few zero-divisors have been introduced in \cite{D}. We define that a ring $R$ has \textit{very few zero-divisors}, if $Z(R)$ is a finite union of prime ideals in $\Ass(R)$. In this section, we prove that if $R$ is a ring that has very few zero-divisors and $B$ is a content $R$-algebra, then $B$ has very few zero-divisors also.

Another celebrated property of Noethering rings is that every ideal entirely contained in the set of its zero-divisors has a nonzero annihilator. A ring $R$ has \textit{Property (A)}, if each finitely generated ideal $I \subseteq Z(R)$ has a nonzero annihilator \cite{HK}. In Section 5, also we prove some results for content algebras over rings having Property (A) and then we discuss on rings and modules having few zero-divisors in more details.

$ $

In Section 5, we discuss Gaussian and Armendariz content algebras that are natural generalization of the same concepts in polynomials rings. In this section we show that if $B$ is a content $R$-algebra, then $B$ is a Gaussian $R$-algebra iff for any ideal $I$ of $R$, $B/IB$ is an Armendariz $(R/I)$-algebra. This is a generalization of a result in \cite{AC}.

$ $

In Section 6, we prove some results about Nilradical and Jacobson radical of content algebras, i.e., statements about content algebras over domainlike and presimplifiable rings. Also we show some results similar to what we have for the ring $R(X)=R[X]_S$, where $S = \lbrace f \in R[X] \colon c(f) = R \rbrace$.

Some of the results of the present paper can be generalized to \textit{monoid modules}. Whenever it is possible, we bring those results, though their similar proofs are omitted. Unless otherwise stated, our notation and terminology will follow as closely as possible that of Gilmer \cite{G1}.

\section{Content Modules and Algebras}

In this section, first we give the definition of content modules and prove Nakayama lemma for them and then we introduce content and weak content algebras with some results that we need in the rest of the paper for the convinience of the reader. More on content modules and algebras can be found in \cite{OR} and \cite{R}.

\begin{definition}
 Let $R$ be a commutative ring with identity, and $M$ a unitary $R$-module and the \textit{content function}, $c$ from $M$ to the ideals of $R$ defined by
$$
c(x) = \bigcap \lbrace I \colon I \text{~is an ideal of~} R \text{~and~} x \in IM \rbrace.
$$

$M$ is called a \textit{content $R$-module} if $x \in c(x)M $, for all $x \in M$, also when $N$ is a non-empty subset of $M$, then by $c(N)$ we mean the ideal generated by all $c(x)$ that $x \in N$.

\end{definition}

\begin{theorem}

\textbf{Nakayama Lemma for Content Modules}: Let $M$ be a content $R$-module and $\Jac(R)$ be the Jacobson radical of $R$ and $I$ be an ideal of $R$ such that $I \subseteq \Jac(R)$. If $IM = M$, then $M = (0)$.

\end{theorem}

\begin{proof}
Let $x \in M$. Since $M$ is a content $R$-module, $x \in c(x)M $, but $ IM = M$, so $x \in c(x)IM$ and therefore $c(x) \subseteq c(x)I$, but $c(x)$ is a finitely generated ideal of $R$ \cite[1.2, p. 51]{OR}, so by Nakayama lemma for finitely generated modules, $c(x) = (0)$ and consequently $x = 0$.
\end{proof}

\begin{corollary}
 Let $P$ be a projective $R$-module and $\Jac(R)$ be the Jacobson radical of $R$ and $I$ be an ideal of $R$ such that $I \subseteq \Jac(R)$. If $ IP = P$, then $P = (0)$.
\end{corollary}

\begin{proof}
 Any projective module is a content module \cite[Corollary 1.4]{OR}.
\end{proof}

\begin{lemma}
 Let $M$ be an $R$-module. The following statements are equivalent:

\begin{enumerate}
 \item $M$ is a content $R$-module, i.e. $x \in c(x)M $, for all $x \in M$.
 \item For any non-empty family of ideals $\lbrace I_i \rbrace$ of $R$, $\bigcap (I_i) M = \bigcap (I_i M)$.
\end{enumerate}
Moreover when $M$ is a content $R$-module, $c(x)$ is a finitely generated ideal of $R$, for all $x \in M$.
\end{lemma}

\begin{theorem}

 Let $M$ be a content $R$-module. Then the following are equivalent:

\begin{enumerate}
 \item $M$ is flat.
 \item For every $r \in R$ and $x \in M$, $rc(x)=c(rx)$.
\end{enumerate}

Moreover, $M$ is faithfully flat iff $M$ is flat and $c(M)=R$.

\end{theorem}

\begin{proof}
 Proof at \cite[Corollary 1.6, p. 53]{OR} and \cite[Remark 2.3(d), p. 56]{OR}.
\end{proof}

The application of the above theorem will appear in the next section on content algebras. Also with the help of the above theorem, we will describe some of the prime and primary submodules of faithfully flat and content modules.

\begin{definition}
 Let $M$ be an $R$-module and $P$ be a proper $R$-submodule of $M$. $P$ is said to be a \textit{prime submodule} of $M$, if $rx \in P$ implies $x \in P$ or $rM \subseteq P$, for each $r \in R$ and $x \in M$.

\end{definition}

\begin{definition}
 Let $M$ be an $R$-module and $P$ be a proper $R$-submodule of $M$. $P$ is said to be a \textit{primary submodule} of $M$, if $rx \in P$ then $x \in P$ or there exists a natural number $n$ such that $r^nM \subseteq P$, for each $r \in R$ and $x \in M$.
\end{definition}

\begin{theorem}
 Let $M$ be a content and faithfully flat $R$-module and $\underline{p}$ be an ideal of $R$. Then $\underline{p}M$ is a primary (prime) $R$-submodule of $M$ iff $\underline{p}$ is a primary (prime) ideal of $R$.
\end{theorem}

\begin{proof}
 Let $\underline{p}$ be a prime ideal of $R$ and $r \in R$ and $x \in M$ such that $rx \in \underline{p}M$. Therefore $c(rx) \subseteq \underline{p}$ and since $c(rx) = rc(x)$ we have $rc(x) \subseteq \underline{p}$ and this means that $c(x) \subseteq \underline{p}$ or $(r) \subseteq \underline{p}$ and at last $x \in \underline{p}M$ or $rM \subseteq \underline{p}M$. Notice that since $M$ is a faithfully flat $R$-module, $\underline{p}M \not= M$. The other assertions can be proved in a similar way.
\end{proof}

Content algebras and later weak content algebras were introduced and discussed in \cite{OR} and \cite{R} respectively. Content algebras are actually a natural generalization of (almost) polynomial rings \cite{ES}. Let $R$ be a commutative ring with identity. For $f \in R[X]$, the content of $f$, denoted by $c(f)$, is defined as the $R$-ideal generated by the coefficients of $f$. One can easily check that $c(fg) \subseteq c(f)c(g)$ for the two polynomials $f, g \in R[X]$ and may ask when the equation $c(fg) = c(f)c(g)$ holds. Tsang, a student of Kaplansky, proved that if $D$ is an integral domain and $c(f)$, for $f \in D[X]$, is an invertible ideal of $D$, then $c(fg) = c(f)c(g)$, for all $g \in D[X]$. Tsang's guess was that the converse was true and the correctness of her guess was completely proved some decades later \cite{LR}. Though the equation $c(fg) = c(f)c(g)$ does not hold always, a weaker formula always holds that is called the \textit{Dedekind-Mertens content formula} \cite{AG}.

\begin{theorem}
 \textbf{Dedekind-Mertens Lemma}\textbf{.} Let $R$ be a ring. For each $f$ and $g$ in $R[X]$, there exists a natural number $n$ such that $c(f)^n c(g) = c(f)^{n-1} c(fg)$.
\end{theorem}

For a history of Dedekind-Mertens lemma, refer to \cite{HH} and its combinatorial proof, refer to \cite[Corollary 2]{BG}. Now we bring the definition of content algebras from \cite[6, p. 63]{OR}:

\begin{definition}
 Let $R$ be a commutative ring with identity and $R^\prime$ an $R$-algebra. $R^\prime$ is defined to be a \textit{content $R$-algebra}, if the following conditions hold:

\begin{enumerate}
 \item
$R^\prime$ is a content $R$-module.
 \item
(\textit{Faithfully flatness}) For any $r \in R$ and $f \in R^\prime$, the equation $c(rf) = rc(f)$ holds and $c(R^\prime) = R$.
 \item
(\textit{Dedekind-Mertens content formula}) For each $f$ and $g$ in $R^\prime$, there exists a natural number $n$ such that $c(f)^n c(g) = c(f)^{n-1} c(fg)$.
\end{enumerate}

\end{definition}

A good example of a content $R$-algebra is the group ring $R[G]$ where $G$ is a torsion-free abelian group \cite{N}. This is actually a free $R$-module. For some examples of content $R$-algebras that as $R$-modules are not free, one can refer to \cite[Examples 6.3, p. 64]{OR}. Rush defined weak content algebras as follows \cite[p. 330]{R}:

\begin{definition}
 Let $R$ be a commutative ring with identity and $R^\prime$ an $R$-algebra. $R^\prime$ is defined to be a \textit{weak content $R$-algebra}, if the following conditions hold:

\begin{enumerate}
 \item
$R^\prime$ is a content $R$-module.
 \item
(\textit{Weak content formula}) For all $f$ and $g$ in $R^\prime$, $c(f)c(g) \subseteq \rad (c(fg))$ (Here $\rad(A)$ denotes the radical of the ideal $A$).
\end{enumerate}

\end{definition}
Also he gave an equivalent condition for when an algebra that is a content module is a weak content algebra \cite[Theorem 1.2, p. 330]{R}:
\begin{theorem}

\label{prime}

 Let $R^ \prime$ be an $R$-algebra such that $R^ \prime$ is a content $R$-module. The following are equivalent:

 \begin{enumerate}
 \item $R^ \prime$ is a weak content $R$-algebra.
 \item For each prime ideal $\underline{p}$ of $R$, either $\underline{p}R^ \prime$ is a prime ideal of $R^ \prime$, or $\underline{p}R^ \prime = R^ \prime$.
\end{enumerate}

\end{theorem}

It is obvious that content algebras are weak content algebras, but the converse is not true. For more on this interesting topic, refer to \cite[Example 4.1]{ESh}. We end our introductory section with the following results:

\begin{theorem}
 Let $R$ be a ring and $S$ be a commutative monoid. Then the following statements about the monoid algebra $B = R[S]$ are equivalent:

\begin{enumerate}
 \item $B$ is a content $R$-algebra.
 \item $B$ is a weak content $R$-algebra.
 \item For $f,g \in B$, if $c(f) = c(g) = R$, then $c(fg) = R$.
 \item (\textit{McCoy's Property}) For $g \in B$, $g$ is a zero-divisor of $B$ iff there exists $r \in R-\lbrace 0 \rbrace$ such that $rg = 0$.
 \item $S$ is a cancellative and torsion-free monoid.
 \end{enumerate}

\end{theorem}

\begin{proof}
 $(1) \rightarrow (2) \rightarrow (3)$ and $(1) \rightarrow (4)$ are obvious (\cite{OR} and \cite{R}). Also according to \cite{N} (5) implies (1). Therefore the proof will be complete if we prove that (3) and also (4) implies (5).

$(3) \rightarrow (5)$: We prove that if $S$ is not cancellative nor torsion-free then (3) cannot hold. For the moment, suppose that $S$ is not cancellative, so there exist $s,t,u \in S$ such that $s+t = s+u$ while $t \not= u$. Put $f = X^s$ and $g = (X^{t}-X^u)$. Then obviously $c(f) = c(g) = R$, while $c(fg) = (0)$. Finally suppose that $S$ is cancellative but not torsion-free. Let $s,t \in S$ be such that $s \not=t$, while $ns = nt$ for some natural $n$. Choose the natural number $k$ minimal so that $ns = nt$. Then we have $0 = X^{ks}-X^{kt} = (X^s-X^t)(\sum_{i=0}^{k-1} X^{(k-i-1)s+it}) $.

Since $S$ is cancellative, the choice of $k$ implies that $ (k-i_1-1)s+i_{1}t \not= (k-i_2-1)s+i_{2}t $ for $0 \leq i_1 < i_2 \leq k-1 $.
Therefore $\sum_{i=0}^{k-1} X^{(k-i-1)s+it} \not= 0$, and this completes the proof.

In a similar way one can prove $(4) \rightarrow (5)$ \cite[p. 82]{G2}.
\end{proof}

\begin{remark}
 Let $S$ be a commutative monoid and $M$ be a nonzero $R$-module. It is trivial that $M[S]$ is an $R[S]$-module. Let $g \in M[S]$ and put $g = m_1s_1+m_2s_2+\cdots+m_ns_n$, where $m_1,\ldots,m_n \in M$ and $s_1,\ldots,s_n \in S$. We define the content of $g$ to be the $R$-submodule of $M$ generated by the coefficients of $g$, i.e. $c(g) = (m_1,\ldots,m_n)$. The following statements are equivalent:
\begin{enumerate}
 \item $S$ is a cancellative and torsion-free monoid.
 \item For all $f \in R[S]$ and $g \in M[S]$, there exists a natural number $k$ such that $c(f)^k c(g) = c(f)^{k-1} c(fg)$.
 \item (\textit{McCoy's Property}) For all $f \in R[S]$ and $g \in M[S]-\lbrace0\rbrace$, if $fg = 0$, then there exists an $m \in M-\lbrace 0 \rbrace$ such that $f \cdot m = 0$.
\end{enumerate}

\end{remark}

\begin{proof}
 $(1) \rightarrow (2)$ and $(2) \rightarrow (3)$ have been proved in \cite{N} and \cite{M} respectively. For $(3) \rightarrow (1)$ use the technique in the previous theorem.
\end{proof}

\section{Prime Ideals in Content Algebras}

Let $B$ be a weak content $R$-algebra such that for all $\underline{m} \in \Max(R)$ (by $\Max(R)$, we mean the maximal ideals of $R$), we have $\underline{m}B \not = B$, then by Theorem \ref{prime}, prime ideals extend to prime ideals. Particularly in $R$-content algebras - that are faithfully flat $R$-modules by definition - primes extend to primes. We recall that when $B$ is a content $R$-algebra, then $g$ is a zero-divisor of $B$, iff there exists an $r \in R- \lbrace 0 \rbrace $ such that $rg = 0$ \cite[6.1, p. 63]{OR}. Now we give the following theorem about associated prime ideals. We assert that by $\Ass_R (M)$, we mean the associated prime ideals of $R$-module $M$.

\begin{theorem}
 Let $B$ be a content $R$-algebra and $M$ a nonzero $R$-module. If $\underline{p} \in \Ass_R (M)$ then $\underline{p}B \in \Ass_B (M \otimes_R B)$.
\end{theorem}

\begin{proof}
 Let $\underline{p} \in \Ass_R (M)$, therefore $ 0 \longrightarrow R/\underline{p} \longrightarrow M$ is an $R$-exact sequence. Since $B$ is a faithfully flat $R$-module, we have the following $B$-exact sequence:

$$ 0 \longrightarrow B/\underline{p}B \longrightarrow M \otimes_R B$$

with $\underline{p}B = \Ann (x \otimes_R 1_B)$. Since $B$ is a content $R$-algebra, $\underline{p}B$ is a prime ideal of $B$.
\end{proof}

We give a general theorem on minimal prime ideals in algebras. One of the results of this theorem is that in faithfully flat weak content algebras (including content algebras), minimal prime ideals extend to minimal prime ideals and more precisely, there is actually a correspondence between the minimal prime ideals of the ring and their extensions in the algebra.

\begin{theorem}

\label{min}

 Let $B$ be an $R$-algebra with the following properties:

\begin{enumerate}
 \item For each prime ideal $\underline{p}$ of $R$, the extended ideal $\underline{p}B$ of $B$ is prime.
 \item For each prime ideal  $\underline{p}$ of $R$, $\underline{p}B \cap R = \underline{p}$.
\end{enumerate}

Then the function $ \phi : \Min(R) \longrightarrow \Min(B)$ given by $\underline{p} \longrightarrow \underline{p}B$ is a bijection.

\begin{proof}
First we prove that if $\underline{p}$ is a minimal prime ideal of $R$, then $\underline{p}B$ is also a minimal prime ideal of $B$. Let $Q$ be a prime ideal of $B$ such that $Q \subseteq \underline{p}B$. So $Q \cap R \subseteq \underline{p}B \cap R = \underline{p}$. Since $\underline{p}$ is a minimal prime ideal of $R$, we have $Q \cap R =\underline{p}$ and therefore $ Q = \underline{p}B $. This means that $ \phi $ is a well-defined function. Obviously the second condition causes $ \phi $ to be one-to-one. The next step is to prove that $ \phi $ is onto. For showing this, consider $Q \in \Min(B)$, so $Q \cap R$ is a prime ideal of $R$ such that $(Q \cap R)B \subseteq Q$ and therefore $(Q \cap R)B = Q$. Our claim is that $(Q \cap R)$ is a minimal prime ideal of $R$. Suppose $\underline{p}$ is a prime ideal of $R$ such that $ \underline{p} \subseteq Q \cap R$, then $\underline{p}B \subseteq Q$ and since $Q$ is a minimal prime ideal of $B$, $\underline{p}B = Q = (Q \cap R)B$ and therefore $\underline{p} = Q \cap R$.
\end{proof}

\end{theorem}

\begin{corollary}
 Let $B$ be a weak content and faithfully flat $R$-algebra, then the function $ \phi : \Min(R) \longrightarrow \Min(B)$ given by $\underline{p} \longrightarrow \underline{p}B$ is a bijection.
\end{corollary}

\begin{proof}
Since $B$ is a weak content and faithfully flat $R$-algebra, then for each prime ideal $\underline{p}$ of $R$, the extended ideal $\underline{p}B$ of $B$ is prime and also $c(1_B)=R$ by \cite[Corollary 1.6]{OR} and \cite[Theorem 1.2]{R}. Now consider $r \in R$, then $c(r) = c(r\cdot 1_B) = r \cdot c(1_B) = (r)$. Therefore if $r \in IB \cap R$, then $(r) = c(r) \subseteq I$.  Thus for each prime ideal  $\underline{p}$ of $R$, $\underline{p}B \cap R = \underline{p}$.
\end{proof}

\begin{corollary}
 Let $R$ be a Noetherian ring. Then $ \phi : \Min(R) \longrightarrow \Min(R[[X_1,\ldots,X_n]])$ given by $\underline{p} \longrightarrow \underline{p}.(R[[X_1,\ldots,X_n]])$ is a bijection.
\end{corollary}

\section{Rings and modules having few zero-divisors}

For a ring $R$, by $Z(R)$, we mean the set of zero-divisors of $R$. In \cite{D}, it has been defined that a ring $R$ has \textit{few zero-divisors}, if $Z(R)$ is a finite union of prime ideals. We present the following definition to prove some other theorems related to content algebras.

 \begin{definition}
  A ring $R$ has \textit{very few zero-divisors}, if $Z(R)$ is a finite union of prime ideals in $\Ass(R)$.
 \end{definition}

\begin{theorem}

\label{vfzd}

 Let $R$ be a ring that has very few zero-divisors. If $B$ is a content $R$-algebra, then $B$ has very few zero-divisors also.
\end{theorem}

\begin{proof}
 Let $Z(R) = \underline{p_1}\cup \underline{p_2}\cup \cdots \cup \underline{p_n}$, where $\underline{p_i} \in \Ass_R(R)$ for all $1 \leq i \leq n$. We will show that $Z(B) = \underline{p_1}B\cup \underline{p_2}B\cup \cdots \cup \underline{p_n}B$. Let $g \in Z(B)$, so there exists an $r \in R- \lbrace 0 \rbrace $ such that $rg = 0$ and so $rc(g) = (0)$. Therefore $c(g) \subseteq Z(R)$ and this means that $c(g) \subseteq \underline{p_1}\cup \underline{p_2}\cup \cdots \cup \underline{p_n}$ and according to Prime Avoidance Theorem, we have $c(g) \subseteq \underline{p_i}$, for some $1 \leq i \leq n$ and therefore $g \in \underline{p_i}B$. Now let $g \in \underline{p_1}B\cup \underline{p_2}B\cup \cdots \cup \underline{p_n}B$ so there exists an $i$ such that $g \in \underline{p_i}B$, so $c(g) \subseteq \underline{p_i}$ and $c(g)$ has a nonzero annihilator and this means that g is a zero-divisor of $B$. Note that $\underline{p_i}B \in \Ass_B(B)$, for all $1 \leq i \leq n$.
\end{proof}

\begin{remark}
Let $R$ be a ring and consider the following three conditions on $R$:

\begin{enumerate}
 \item $R$ is a Noetherian ring.
 \item $R$ has very few zero-divisors.
 \item $R$ has few zero-divisors.
\end{enumerate}

Then, $(1) \rightarrow (2) \rightarrow (3)$ and none of the implications are reversible.
\end{remark}

\begin{proof}
For  $(1) \rightarrow(2) $ use \cite[p. 55]{K}. It is obvious that $(2) \rightarrow (3)$.

Suppose $k$ is a field, $A=k[X_1, X_2, X_3,\ldots,X_n,\ldots]$ and $\underline{m} =(X_1, X_2, X_3,\ldots, X_n,\ldots)$ and at last $I=(X_1^2, X_2^2, X_3^2,\ldots, X_n^2,\ldots)$. Since $A$ is a content $k$-algebra and $k$ has very few zero-divisors, $A$ has very few zero-divisors while it is not a Noetherian ring. Also consider the ring $R=A/I$. It is easy to check that $R$ is a quasi-local ring with the only prime ideal $\underline{m}/I$ and $Z(R)=\underline{m}/I$ and finally $\underline{m}/I\notin \Ass_R(R)$. Note that $\Ass_R(R)=\emptyset$.
\end{proof}

Now we bring the following definition from \cite{HK} and prove some other results for content algebras.

\begin{definition}
  A ring $R$ has \textit{Property (A)}, if each finitely generated ideal $I \subseteq Z(R)$ has a nonzero annihilator.
 \end{definition}

Let $R$ be a ring. If $R$ has very few zero-divisors (for example if $R$ is Noetherian), then $R$ has Property (A) \cite[Theorem 82, p. 56]{K}, but there are some non-Noetherian rings which have not Property (A) \cite[Exercise 7, p. 63]{K}. The class of non-Noetherian rings having Property (A) is quite large \cite[p. 2]{H}.

\begin{theorem}
 Let $B$ be a content $R$-algebra such that $R$ has Property (A). Then $T(B)$ is a content $T(R)$-algebra, where by $T(R)$, we mean total quotient ring of $R$.
\end{theorem}

\begin{proof}
 Let $S^ \prime = B-Z(B)$. If $S = S^ \prime \cap R$, then $S = R-Z(R)$. We prove that if $c(f) \cap S = \emptyset $, then $f \not\in S^ \prime$. In fact when $c(f) \cap S = \emptyset $, then $c(f) \subseteq Z(R)$ and since $R$ has Property (A), $c(f)$ has a nonzero annihilator. This means that $f$ is a zero-divisor of $B$ and according to \cite[Theorem 6.2, p. 64]{OR}, the proof is complete.
\end{proof}

\begin{theorem}
 Let $B$ be a content $R$-algebra such that the content function $c: B \longrightarrow \FId(R)$ is onto, where by $\FId(R)$, we mean the set of finitely generated ideals of $R$. The following statements are equivalent:

\begin{enumerate}
 \item $R$ has Property (A).
 \item For all $f \in B$, $f$ is a regular member of $B$ iff $c(f)$ is a regular ideal of $R$.
\end{enumerate}

\begin{proof}
 $(1) \rightarrow (2)$: Let $R$ has Property (A). If $f \in B$ is regular, then for all nonzero $r \in R$, $rf \not= 0$ and so for all nonzero $r \in R$, $rc(f) \not= (0)$, i.e. $\text{Ann}(c(f)) = (0)$ and according to the definition of Property (A), $c(f) \not\subseteq Z(R)$. This means that $c(f)$ is a regular ideal of $R$. Now let $c(f)$ be a regular ideal of $R$, so $c(f) \not\subseteq Z(R)$ and therefore $\text{Ann}(c(f)) = (0)$. This means that for all nonzero $r \in R$, $rc(f) \not= (0)$, hence for all nonzero $r \in R$, $rf \not= 0$. Since $B$ is a content $R$-algebra, $f$ is not a zero-divisor of $B$.

$(2) \rightarrow (1)$: Let $I$ be a finitely generated ideal of $R$ such that $I \subseteq Z(R)$. Since the content function $c: B \longrightarrow \FId(R)$ is onto, there exists an $f \in B$ such that $c(f) = I$. But $c(f)$ is not a regular ideal of $R$, therefore according to our assumption, $f$ is not a regular member of $B$. Since $B$ is a content $R$-algebra, there exists a nonzero $r \in R$ such that $rf = 0$ and this means that $rI = (0)$, i.e. $I$ has a nonzero annihilator.
\end{proof}

\begin{remark}
 In the above theorem the surjectivity condition for the content function $c$ is necessary, because obviously $R$ is a content $R$-algebra and the condition (2) is satisfied, while one can choose the ring $R$ such that it does not have Property (A) \cite[Exercise 7, p. 63]{K}.
\end{remark}

\end{theorem}

\begin{theorem}
 Let $R$ have property (A) and $B$ be a content $R$-algebra. Then $Z(B)$ is a finite union of prime ideals in $\Min(B)$ iff $Z(R)$ is a finite union of prime ideals in $\Min(R)$.
\end{theorem}

\begin{proof}
 The proof is similar to the proof of Theorem \ref{vfzd} by considering Theorem \ref{min}.
\end{proof}

Please note that if $R$ is a Noetherian reduced ring, then $Z(R)$ is a finite union of prime ideals in $\Min(R)$ (Refer to \cite[Theorem 88, p. 59]{K} and \cite[Corollary 2.4]{H}). Now we generalize the definition of rings having very few zero-divisors in the following way and prove the \textit{monoid module} version of the above theorem.

\begin{definition}
  An $R$-module $M$ has \textit{very few zero-divisors}, if $Z_R(M)$ is a finite union of prime ideals in $\Ass_R(M)$.
 \end{definition}

\begin{remark}
 \textit{Examples of modules having very few zero-divisors}. If $R$ is a Noetherian ring and $M$ is an $R$-module such that $\Ass_R(M)$ is finite, then obviously $M$ has very few zero-divisors. For example $\Ass_R(M)$ is finite if $M$ is a finitely generated $R$-module \cite[p. 55]{K}. Also if $R$ is a Noetherian quasi-local ring and $M$ is a balanced big Cohen-Macaulay $R$-module, then $\Ass_R(M)$ is finite \cite[Proposition 8.5.5, p. 344]{BH}.
\end{remark}

\begin{theorem}
 Let $R$-module $M$ have very few zero-divisors. If $S$ is a commutative, cancellative, torsion-free monoid then the $R[S]$-module $M[S]$ has very few zero-divisors also.
\end{theorem}

\begin{proof}
 The proof is similar to the proof of Theorem \ref{vfzd}.
\end{proof}

\section{Gaussian and Armendariz Algebras}

\begin{definition}
 Let $B$ be an $R$-algebra that is a content $R$-module. $B$ is said to be a \textit{Gaussian $R$-algebra} if $c(fg) = c(f)c(g)$, for all $f,g \in B$.
\end{definition}

For example if $B$ is a content $R$-algebra such that every nonzero finitely generated ideal of $R$ is cancellation ideal of $R$, then $B$ is a Gaussian $R$-algebra. Another example is given in the following remark:

\begin{remark}
 Let $(R,\underline{m})$ be a quasi-local ring with $\underline{m}^2 = (0)$. If $B$ is a content $R$-algebra, then $B$ is a Gaussian $R$-algebra.

\begin{proof}
 Let $f,g \in B$ such that $c(f) \subseteq \underline{m}$ and $c(g) \subseteq \underline{m}$, then $c(fg) = c(f)c(g) = (0)$, otherwise one of them, say $c(f)$, is $R$ and according to Dedekind-Mertens content formula, we have $c(fg) = c(g) = c(f)c(g)$.
\end{proof}

\end{remark}

\begin{theorem}
 Let $M$ be an $R$-module such that every finitely generated $R$-submodule of $M$ is cyclic and $S$ be a commutative, cancellative, torsion-free monoid. Then for all $f\in R[S]$ and $g \in M[S]$, $c(fg) = c(f)c(g)$.
\end{theorem}

\begin{proof}
 Let $g \in M[S]$ such that $g = m_1g_1+m_2g_2+\cdots+m_ng_n$, where $m_1, m_2,\ldots,m_n \in M$ and $g_1, g_2,\ldots,g_n \in S$. Then there exists an $m \in M$, such that $c(g) = (m_1, m_2,\ldots,m_n) = (m)$. From this, we can get $m_i = r_im$ and $m = \sum s_im_i$, where $r_i, s_i \in R$. Put $d = \sum s_ir_i$, then $m = dm$. Since $S$ is an infinite set, it is possible to choose $g_{n+1} \in S- \lbrace g_1,g_2,\ldots,g_n \rbrace $ and put $g^ \prime = r_1g_1+r_2g_2+\cdots+r_ng_n+(1-d)g_{n+1}$. One can easily check that $g =g^\prime m$ and $c(g^ \prime) =R$ and
$c(fg) = c(fg^ \prime m) = c(fg^ \prime)m =c(f)m = c(f)c(g). $
\end{proof}

\begin{corollary}
Let $R$ be a ring such that every finitely generated ideal of $R$ is principal and $S$ be a commutative, cancellative, torsion-free monoid. Then $R[S]$ is a Gaussian $R$-algebra.
\end{corollary}

For more about content formulas for polynomial modules, refer to \cite{NY} and \cite{AK}.

$ $

In the next step, we define Armendariz algebras and show their relationships with Gaussian algebras. Armendariz rings were introduced in \cite{RC}. A ring $R$ is said to be an Armendariz ring if for all $f,g \in R[X]$
with $f = a_0+a_1X+\cdots+a_nX^n$ and $g = b_0+b_1X+\cdots+b_mX^m$, $fg = 0$ implies $a_ib_j = 0$, for all $0 \leq i \leq n$ and $0 \leq j \leq m$. This is equivalent to say that if $fg = 0$, then $c(f)c(g) = 0$ and our inspiration to define Armendariz algebras.

\begin{definition}
 Let $B$ be an $R$-algebra such that it is a content $R$-module. We say $B$ is an \textit{Armendariz $R$-algebra} if for all $f,g \in B$, if $fg = 0$, then $c(f)c(g) = (0) $.
\end{definition}

For example if $B$ is a weak content $R$-algebra and $R$ is a reduced ring, then $B$ is an Armendariz $R$-algebra.

\begin{theorem}
 Let $R$ be a ring and $(0)$ a $\underline{p}$-primary ideal of $R$ such that $\underline{p}^2 = (0)$ and $B$ a content $R$-algebra. Then $B$ is an Armendariz $R$-algebra.

\begin{proof}
 Let $f,g \in B$, where $fg = 0$. If $f = 0$ or $g = 0$, then definitely $c(f)c(g) = 0$, otherwise suppose that $f \not= 0$ and $g \not= 0$, therefore $f$ and $g$ are both zero-divisors of $B$. Since $(0)$ is a $\underline{p}$-primary ideal of $R$, so $(0)$ is a $\underline{p}B$-primary ideal of $B$ \cite[R, p. 331]{R} and therefore $\underline{p}B$ is the set of zero-divisors of $B$. So $f,g \in \underline{p}B$ and this means that $c(f) \subseteq \underline{p}$ and $c(g) \subseteq \underline{p}$. Finally $c(f)c(g) \subseteq \underline{p}^2 = (0)$.
\end{proof}

\end{theorem}

In order to characterize Gaussian algebras in terms of Armendariz algebras, we should mention the following useful remark.

\begin{remark}
 Let $R$ be a ring and $I$ an ideal of $R$. If $B$ is a Gaussian $R$-algebra then $B/IB$ is a Gaussian $(R/I)$-algebra also.
\end{remark}

\begin{theorem}
 Let $B$ be a content $R$-algebra. Then $B$ is a Gaussian $R$-algebra iff for any ideal $I$ of $R$, $B/IB$ is an Armendariz $(R/I)$-algebra.

\begin{proof}
 $ ( \rightarrow ): $ According to the above remark, since $B$ is a Gaussian $R$-algebra, $B/IB$ is a Gaussian $(R/I)$-algebra and obviously any Gaussian algebra is an Armendariz algebra and this completes the proof.

$ ( \leftarrow ): $ One can easily check that if $B$ is an algebra such that it is a content $R$-module, then for all $f,g \in B$, $c(fg) \subseteq c(f)c(g)$ \cite[Proposition 1.1, p. 330]{R}. Therefore we need to prove that $c(f)c(g) \subseteq c(fg)$. Put $I = c(fg)$, since $B/IB$ is an Armendariz $(R/I)$-algebra and $c(fg+IB) = I$ so $c(f+IB)c(g+IB) = I$ and this means that $c(f)c(g) \subseteq c(fg)$.
\end{proof}

\end{theorem}

For more about Armendariz and Gaussian rings, one can refer to \cite{AC}.

\section{Nilradical and Jacobson radical of content algebras}

\begin{definition}
 A ring $R$ is said to be \textit{domainlike} if any zero-divisor of $R$ is nilpotent, i.e. $Z(R) \subseteq \Nil(R)$ \cite[Definition 9]{AAFS}.
\end{definition}

\begin{theorem}
 If $B$ is a content $R$-algebra, then $R$ is domainlike iff $B$ is domainlike.
\end{theorem}

\begin{proof}
 The ring $R$ is domainlike iff $(0)$ is a primary ideal of $R$ \cite[Lemma 10]{AAFS}. Also it is easy to prove that if $B$ is a content $R$-algebra, then $q$ is a $p$-primary ideal of $R$ iff $qB$ is a $pB$-primary ideal of $B$ [R, p. 331].
\end{proof}

In a similar way one can see:
\begin{remark}
 Let $S$ be a commutative, cancellative and torsion-free monoid and $M$ be an $R$-module. Then $Z_R(M) \subseteq \Nil(R)$ iff $Z_{R[S]} (M[S]) \subseteq \Nil(R[S])$.
\end{remark}

\begin{remark}
 If $B$ is a weak content $R$-algebra, then $\Nil(B) = \Nil(R)B$, particularly $R$ is a reduced ring iff $B$ is a reduced ring.
\end{remark}

\begin{proof}
 It is obvious that $\Nil(R)B \subseteq \Nil(B)$. Also it is easy to prove that for all $f \in B$ and natural number $n$, we have $c(f)^n \subseteq \text{rad}(c(f^n))$ and therefore if $f \in B$ is nilpotent, then $c(f) \subseteq \Nil(R)$ and consequently $f \in \Nil(R)B$.
\end{proof}

\begin{definition}
A ring $R$ is called \textit{presimplifiable} if any zero-divisor of $R$ is a member of the Jacobson radical of $R$, i.e. $ Z(R) \subseteq \Jac(R) $.
\end{definition}

For more about presimplifiable rings, one can refer to \cite{AAFS}. In the following, our aim is show when some of the content algebras are presimplifiable. For doing that, we need to know about localization of content algebras that have been discussed in \cite[Section 3, p. 56-58]{OR}. Actually we are interested in the following special case of localization:

Let $B$ be a content $R$-algebra and $S^ \prime = \lbrace f \in B \colon c(f) = R \rbrace$. It is easy to check that $ S^ \prime = B-\bigcup_{\underline{m}\in \Max(R)} \underline{m}B $ and $S = S^ \prime \cap R = U(R)$, where by $U(R)$, we mean the units of $R$. According to \cite[Theorem 6.2, p. 64]{OR}, it is clear that $B_{S^ \prime} $ is also a content $R$-algebra and $B$ is a subring of $B_{S^ \prime}$. This special content $R$-algebra has some interesting properties:

\begin{theorem}

 Let $B$ be a content $R$-algebra such that $S^ \prime = \lbrace f \in B \colon c(f) = R \rbrace$ and put $R^ \prime = B_{S^ \prime} $, then the following statements hold:
\begin{enumerate}
 \item The map $\phi: \Max(R) \longrightarrow \Max(R^ \prime)$, defined by $I \longrightarrow IR^ \prime$ is a bijection.
\item $\Jac(R^ \prime) = \Jac(R)R^ \prime$.
\item $U(R^ \prime) = \lbrace f/g \colon c(f)=c(g)=R \rbrace$
\item The ring $R^ \prime$ is presimplifiable iff $R$ is presimplifiable.
\item The ring $R^ \prime$ is $0$-dimensional iff $R$ is $0$-dimensional.
\end{enumerate}

\end{theorem}

\begin{proof}
The first proposition is actually a special case of \cite[4.8]{G1}. For the proof of the second proposition, notice that the Jacobson radical of a ring is the intersection of all maximal ideals. Now use \cite[1.2, p. 51]{OR}.

It is obvious that if $c(f)=c(g)=R$, then $f/g$ is a unit of $R ^ \prime$. Now let $f/g$ be a unit of $R ^ \prime$, where $c(g)=R$ and assume that there exists a member of $R^ \prime$, say $f^ \prime/g^ \prime$ with $c(g^ \prime)=R$, such that $(f/g)\cdot (f^ \prime/g^ \prime) = 1/1$. According to McCoy's property for content algebras, $S^ \prime \subseteq B- Z_B(B)$. So $ff^\prime = gg^ \prime$ and $ff^\prime \in S^\prime$. This means that $f \in S^\prime$ and the proof of the third proposition is complete.

For the proof of the forth proposition, suppose $R$ is presimplifiable and let $f \in Z(R^ \prime)$. Therefore there exists a nonzero $r \in R$ such that $rf=0$ and so $rc(f)=(0)$. This means that $c(f) \subseteq Z(R)$. Since $R$ is presimplifiable, $c(f) \subseteq \Jac(R)$ and at last $f \in \Jac(R)R^ \prime$ and according to (2), $f \in \Jac(R^ \prime)$. It is easy to check that if $R^ \prime$ is presimplifiable then $R$ is presimplifiable also. For the proof of the fifth proposition note that a ring, say $T$, is $0$-dimensional iff $\Min(T)=\Max(T)$.
\end{proof}

\begin{theorem}
 Let $B$ be a content $R$-algebra with the property that if $f \in B$ with $c(f) = (a)$ where $a\in R$, then there exists an $f_1 \in B$ such that $f=af_1$ and $c(f_1) = R$ and put $S^ \prime = \lbrace f \in B \colon c(f) = R \rbrace$ and $R^ \prime = B_{S^ \prime} $. Then the idempotent members of $R$ and $R^ \prime$ coincide.
\end{theorem}

\begin{proof}

Let $f/g$ be an idempotent member of $R^ \prime$, where $f,g \in B$ and $c(g) =R$. Therefore $fg^2 = gf^2$ and since $g$ is a regular member of $B$, we have $fg = f^2$. So $c(f^2) = c(fg) = c(f)$, but $c(f^2) \subseteq c(f)^2$, therefore $c(f)^2 = c(f)$. We know that every finitely generated idempotent ideal of a ring is generated by an idempotent member of the ring [G1, p. 63]. Therefore we can suppose that $c(f) = (e)$ such that $e^2 = e$. On the other side we can find an $f_1 \in B$ such that $f = ef_1$ and $c(f_1) = R$. Consider $ef_1 /g = f/g = f^2/g^2 = e^2f_1^2/g^2$. Since $f_1$ and $g$ are both regular, and $e$ is idempotent, we have $e = ef_1/g = f/g \in R$.
\end{proof}

\begin{corollary}
Let $R$ be a ring and $M$ a commutative, cancellative and torsion-free monoid and put $S^ \prime=\lbrace f \in R[M] : c(f) = R \rbrace$ and $R^ \prime = B_{S^ \prime}$. Then the idempotent members of $R$ and $R^ \prime$ coincide.
\end{corollary}

\begin{definition}
 A commutative ring $R$ is said to be a \textit{valuation ring} if for any $a$ and $b$ in $R$ either $a$ divides $b$ or $b$ divides $a$ (\cite[p. 35]{K}).
\end{definition}

\begin{theorem}
 Let $B$ be a content $R$-algebra with the property that if $f \in B$ with $c(f) = (a)$ where $a\in R$, then there exists an $f_1 \in B$ such that $f=af_1$ and $c(f_1) = R$ and put $S^ \prime = \lbrace f \in B \colon c(f) = R \rbrace$ and $R^ \prime = B_{S^ \prime} $. If $R$ is a valuation ring, then so is $R^ \prime = B_{S^ \prime} $.
\end{theorem}

\begin{proof}

Let $f/g$ be a member of $R^ \prime$, where $f,g \in B$ and $c(g) =R$. Since $c(f)$ is a finitely generated ideal of $R$ and $R$ is a valuation ring, there exists an $r \in R$ such that $c(f) = (r)$ and therefore there exists an $f_1 \in B$ such that $f=rf_1$ and $c(f_1) = R$. By considering this fact that $f_1/g$ is a unit in $R^ \prime$, it is obvious that $R^ \prime$ is also a valuation ring and the proof is complete.
\end{proof}

\section{Acknowledgment} The author wishes to thank Prof. Winfried Bruns for his useful advice.

\end{document}